\documentclass{amsart}
\usepackage{blindtext,inputenc,graphicx}
\usepackage{amsmath,amssymb,amsthm,bm,mathrsfs}
\usepackage{hyperref}
\usepackage{lineno}
\usepackage{xcolor}
\usepackage{marginnote}

\hypersetup{
    colorlinks=true,
    linkcolor=blue,
    filecolor=magenta,      
    urlcolor=cyan,
}
\theoremstyle{definition}
\newtheorem{defn}{Definition}[section]
\newtheorem{theorem}{Theorem}[section]
\newtheorem{prop}{Proposition}[section]
\newtheorem{lemma}{Lemma}[section]
\newtheorem{remark}{Remark}[section]

\DeclareMathOperator{\CC}{\mathbb{C}}
\DeclareMathOperator{\RR}{\mathbb{R}}

\DeclareMathOperator{\OO}{\mathcal{O}}

\renewcommand{\Re}{\operatorname{Re}}

\numberwithin{equation}{section}
\numberwithin{figure}{section}

\begin{document}

\title[A Note on Electrified Droplets]{A Note on Electrified Droplets} 

\date{\today.}

\author[N. Hayford]{Nathan Hayford}
\email{nhayford@usf.edu}
\address{Department of Mathematics and Statistics, University of South Florida, Tampa, FL 33620 USA}

\author[F. Wang]{Fudong Wang}
\email{fudong@math.ucf.edu}
\address{Department of Mathematics, University of Central Florida, Orlando, FL 32816 USA}

\begin{abstract}
We give an in-depth analysis of a $1$-parameter family of electrified droplets first described in \cite{KSV}. We
also investigate a technique for searching for new solutions to the droplet equation, and rederive via this technique a 
$1$ parameter family of physical droplets, which were first discovered by Crowdy \cite{DC}. We speculate on extensions of 
these solutions, in particular to the case of a droplet with multiple connected components.
\end{abstract}

\maketitle

\section{Introduction.}
	Consider a droplet of a perfectly conducting fluid, which has come to equilibrium with an external electric 
	field, with analytic potential (meaning the analytic completion of the potential due to the field) $g(z)$ given by 
		\begin{equation}
			g(z) = E_\infty z + a_0 + \frac{a_1}{z} + \cdots
		\end{equation}
	By making a global rotation, dilation, and translation if necessary, it can be assumed that $E_\infty = 1$,
	$a_0 = 0$. If we consider only the forces of pressure and surface tension (in addition to the force of 
	the external field) to be significant, then the boundary of the droplet must satisfy the equation (cf. \cite{G})
		\begin{equation}\label{dropleteq}
			F(z) = p \bar{z} + i \tau \dot{\bar{z}},
		\end{equation}
	where ``dot" here represents the derivative with respect to arc length, and for  positive constants $p$ 
	(representing the strength of the pressure force), $\tau$ (representing the strength of the
	force due to surface tension), and with the function $F(z)$ related to the analytic potential by
		\begin{equation}
			F(z) = \int (g'(z))^2 dz = z + \frac{2a_1}{z} + \cdots
		\end{equation}
	Following \cite{KSV}, we will refer to $F(z)$ as the \textit{integrated analytic potential}; this is a sensible name, since
	$\Re g(z)$ is the potential of the electric field, and $g(z)$ its analytic completion. For $F$ 
	corresponding to a single-valued electric field, necessarily, we must have that $\sqrt{F'(z)}$ is 
	single-valued. Given $p, \tau$, and $F$ analytic in $\Omega$ (here, $\Omega$ is the complement of the droplet in 
	$\CC$) with a simple pole at infinity, a solution to \eqref{dropleteq} is called a \textit{physical droplet} 
	if $\sqrt{F'(z)}$ is single-valued, and is otherwise called a \textit{mathematical droplet}. 
	
	Until 1999, the only known examples of physical droplets were the circles, as well as a solution discovered in 1955 by E. McLeod \cite{M}, 
	which is described by the conformal map from the disc to the exterior of the droplet:
	\begin{equation}
	    \varphi(w) = \frac{1}{w} + \frac{2}{3}w - \frac{1}{27}w^3,
	\end{equation}
	with the following equation holding on the boundary of the droplet:
	    \begin{equation}
	        F(z) = i\tau \dot{\bar{z}},
	    \end{equation}
	with $\tau = 3$ and $F(z) = \frac{3w^2 + 1}{w(1+w^2/3)}$, and $w = \varphi^{-1}(z)$. In the fluid mechanics community, several solutions 
	to a closely related problem (the so-called ``deep-water capillary wave'' problem) were discovered \cite{Crapper,Kinnersley}, however, no new droplet
	solutions were identified. The direct connection between this problem and the one discussed here may be found in \cite{DC-00}. This connection 
	enabled Crowdy \cite{DC} in 1999 to discover a new 1-parameter family of physical droplets, albeit in a different physical context 
	\footnote{The mathematical content of the problem considered is identical to that considered here; for the direct analogy of these problems, 
	cf. \cite{DC-15}, for example.}; these droplets were also simultaneously discovered (independently) by N. Zubarev \cite{Z1}. In \cite{KSV}, 
	Khavinson et. al. constructed a 1-parameter family of mathematical droplets, containing the McLeod droplet as a special case. However, no new
	physical droplets were discovered. Since then, some progress in finding additional solutions has been made \cite{Z2}, including solutions 
	with additional symmetry properties \cite{WC}. These solutions are the ones which we shall re-derive via the methods we outline below.
    	
	Our investigations are based on two major tools from complex analysis: those
	of quadratic differentials and some facts from the theory of Smirnov classes and Smirnov domains. 
	For the former, we refer the reader to either \cite{S} for a general introduction, or to Appendix \ref{A}, 
	where all facts about quadratic differentials that we need are laid out. For the latter, we refer the reader to the 
	book \cite{D}, or the paper \cite{K}. We briefly overview the basic facts about Smirnov classes and domains
	that will be needed here.
	\begin{defn}
	    Let $\Omega$ be a simply connected domain with smooth boundary. We say an analytic function $f$ belongs to the
	    Smirnov class $E^1(\Omega)$ if $f$ has non-tangential boundary values $f^* \in L^1(\partial \Omega)$ a.e. on the
	    boundary, and if $f$ can be recovered from its boundary values via the Cauchy integral:
	        \begin{equation}
	            f(z) = \int_{\partial \Omega} \frac{f^*(\zeta)}{\zeta - z} \frac{d\zeta}{2\pi i}, \qquad z \in \Omega.
	        \end{equation}
	\end{defn}
	\begin{defn}
	    A simply connected domain $\Omega$ with smooth boundary is called \textit{Smirnov} if, for any conformal map $\varphi(\zeta)$ from $\mathbb{D}$ to $\Omega$, and for any $\zeta \in \mathbb{D}$,
	        \begin{equation}
	            \log |\varphi'(\zeta)| = \frac{1}{2\pi} \int_{\partial \mathbb{D}} \log|\varphi'(w)| 
	            \frac{\partial g_{\mathbb{D}}(w,\zeta)}{\partial n} |dw|,
	        \end{equation}
	   where $g_{\mathbb{D}}(w,\zeta)$ is the Green function for $\mathbb{D}$ with logarithmic singularity at $\zeta$.
	\end{defn}
    In other words, $\varphi'$ can be recovered from the boundary values of $|\varphi'|$ via the Poisson integral 
	formula. This property of Smirnov domains will be used frequently. The above definition should be modified appropriately if 
	$\Omega$ is not simply connected; in this case, the disc should be replaced with the corresponding canonical Koebe domain $G$ 
	(the unit disc with discs excised), and all corresponding objects should also be modified (the conformal map should be from $G$
	to $\Omega$, and the Green's function should be that of $G$).
	
	The outline of the rest of the paper is as follows: in \S2, we briefly review the techniques 
	involved in finding solutions to equation \eqref{dropleteq}; in \S3, we give a detailed analysis of a 
	particular $1$-parameter family of solutions, first discovered in \cite{KSV}. In \S4, we make use of
	the ideas of \S2 to construct a new family of droplets, which generate for us a few new $1$-parameter families
	of solutions to the droplet equation, with the $1$-parameter family from \S2 appearing as a limiting case. 
	We remark that the methods we use generalize, and via a clever symmetry argument found in \cite{WC}, can produce the so-
	called $m$-pole solutions to the droplet equation. In \S5, we give an analysis of the new physical
	droplets of \S4, similar to the analysis of the droplet of \S2 in \S3. Finally, in \S6 we conclude by discussing
	methods for constructing solutions to the droplet equation for droplets with multiple connected components. On this last
	point, we remark that some work has already been done \cite{DC-99,DC-01}, but the study of droplets with multiple connected components,
	and especially the transition from one to two connected components, is still very much an active area of research 
	\cite{DNK-20,DNK-21}.
	
\section{Finding mathematical droplets.}
Let us briefly review the techniques used  in \cite{KSV} to find solutions to \eqref{dropleteq}.  The computations here
are identical to those found in \cite{KSV}, and are reproduced here for the reader's convenience.
The underlying principle for constructing solutions is 
as follows: map the problem onto the disc using a conformal map, and then consider an appropriately chosen 
quadratic differential there. Then, using some tools from the theory of quadratic differentials (see 
Appendix \ref{A}), we can uniquely determine the quadratic differential we constructed, and from it, we 
can reconstruct the conformal map, and thus the droplet. Let $\Omega$ denote the complement of the droplet, 
since it is assumed that the droplet is connected and simply connected, $\Omega$ is simply connected, and we can find a 
conformal map $\varphi: \mathbb{D} \to \Omega$ from the unit disc to $\Omega$, normalized so that $\varphi(0) = \infty$, 
$\text{Res}_{z = 0} \varphi(z) > 0$. We can write $\varphi(w)$ as
	\begin{equation}
		\varphi(w) = R\left(\frac{1}{w} + b_1w + b_2w^2 + \cdots \right),
	\end{equation}
for $R > 0$ (actually, one can show $R = \text{cap}(\gamma)$, the logarithmic capacity of the boundary of 
the droplet $\gamma := \partial \Omega$, but this will not be important to us here). 

We start with a slightly modified version of the droplet equation \eqref{dropleteq}.
Suppose $\gamma$ is a rectifiable Jordan curve, with exterior $\Omega$ being a Smirnov domain. Furthermore, 
let $F$ be in $E^1$ near the boundary of $\Omega$, with
	\begin{equation} \label{modF}
		F(z) = \dot{\bar{z}} \qquad\text{a.e. on $\gamma$.}
	\end{equation}
Suppose further that $F$ has a single pole at a finite point $0 \in \Omega$, and $F(\infty) \neq 0$.
Let $\varphi$ be a conformal map from $\mathbb{D}$ onto $\Omega$ such that $\varphi(0) = \infty$, and
$\varphi(c) = 0$, for some constant $0 < c < 1$.  For such $F$, we have the following proposition:
    \begin{prop}
        Let $F$ satisfy the above hypotheses, and suppose that equation \eqref{modF} holds on the boundary of the droplet.
        Then, there exists a 1-parameter family of mathematical droplets described by the family of conformal maps
            \begin{equation} \label{old-droplet}
                \varphi(w) = w^{-1} + \frac{c}{1-cw} - c^2 w +c_0,
            \end{equation}
        where $c_0$ is chosen so that $\varphi(c) = 0$, satisfying the equation
            \begin{equation}
                G_c(z) = p S(z) + i\tau F(z) = p \bar{z} + i \tau \dot{\bar{z}},
            \end{equation}
        where $p = 1 - c^2$, $\tau = 1 - c^2 + c^4$, and where $S(z) = \bar{z}$ is the Schwarz function \footnote{For further details on the theory of 
Schwarz functions, we refer the reader to \cite{DS,Shapiro}.} of the boundary of the droplet, and 
        $G_c(z)$ is analytic in the complement of the droplet, except for a pole at infinity, and is given in the local 
        coordinate $w = \varphi^{-1}(z)$ by
            \begin{equation}
                \hat{G}_c(w) := G_c(\varphi(w)) = \frac{p_4(w)}{w(c^2w^2-cw+1)},
            \end{equation}
        where $p_4(w)$ is a degree $4$ polynomial in $w$ with coefficients depending on $c$.
        The maps \eqref{old-droplet} are univalent when $0 < c < (\sqrt{5} - 1)/2 = 0.618...$. 
        Furthermore, the droplets are convex for $0 < c < c_1 := (\sqrt{5} - 3)/2 = 0.382...$.
    \end{prop}
We point out a notation that we will make use of frequently; when expressing $G_c(z)$ (or any function, for that matter) in the local
variable $w = \varphi^{-1}(z)$, we write $\hat{G}_c(w) := G_c(\varphi(w))$.
\begin{proof}
We postpone the proofs of univalency of the map $\varphi(w)$ and convexity of the curve in the given parameter range until \S3.
We also will drop the constant $c_0$ from our exposition, which violates the condition $\varphi(c) = 0$, but makes all formulas 
look cleaner. Since one can recover the desired map by a global translation of the droplet, there is no ambiguity. 
Now, we consider the quadratic differential
	\begin{equation} \label{QD2}
		Q := F^2(z) dz^2,	
	\end{equation}
defined on the complement of the droplet. Using the conformal map, we can write
	\begin{equation} \label{mappedQD2}
		Q = Q(w)dw^2:= F^2(\varphi(w)) (\varphi'(w))^2 dw^2,	
	\end{equation}
which is a well-defined quadratic differential on the disc, with a pole of order $2$ at $w = c$, and a pole of 
order $4$ at $w = 0$. Moreover, since $Q = ds^2 > 0$ on the circle, it can be continued by reflection to the whole
Riemann sphere. To summarize, $Q$ is a quadratic differential with the following properties:
	\begin{enumerate}
		\item The only poles of $Q$ in $\mathbb{D}$ are a pole of order $2$ at $w = c$, and a pole of 
		order $4$ at $w = 0$.
		\item $Q$ is positive for $|w| = 1$.
	\end{enumerate}
It follows that $Q$ necessarily takes the form
	\begin{equation} \label{QD2-form}
		Q = -C^2 \frac{(w-A)^2(1-\bar{A}w)^2(w-B)^2(1-\bar{B}w)^2}{(w-c)^2(1-\bar{c}w)^2 w^4}dw^2,
	\end{equation}
for undetermined constants $C > 0$, $|A| < 1$, $|B| < 1$ (\textit{For a sketch of the proof of this fact, see 
Appendix \ref{A}}). Thus, equations \eqref{QD2}, \eqref{mappedQD2}, and \eqref{QD2-form} imply that
	\begin{equation}
		|w^2 \varphi'(w)| = C \frac{|(1-\bar{A}w)(1-\bar{B}w)|^2}{|1-cw|^2} \qquad \text{a.e. $|w| = 1$.}
	\end{equation}
Now, since the droplet has rectifiable boundary, and $\varphi'$ is in class $E^1$ near the boundary, and $\Omega$ is Smirnov domain, we can recover 
$\varphi'(w)$ from its a.e. boundary values, via the Poisson integral formula:
	\begin{equation}
		\varphi'(w) = w^{-2}\frac{(1-\bar{A}w)^2(1-\bar{B}w)^2}{(1-cw)^2}.
	\end{equation}
Integrating, we obtain that:
	\begin{equation}
		 \varphi(w) = C\left( w^{-1} -\alpha_1 \log(w) - \frac{\alpha_2}{1-cw} + \alpha_3 \log(1-cw) 
		 - \alpha_4 w\right) + c_0,
	\end{equation}
for some integration constant $c_0$; since $c_0$ only contributes an overall translation of the droplet and does
not change its shape, we set it to be zero. The constants $\alpha_1$, $\alpha_2$, $\alpha_3$, and $\alpha_4$ are:
\begin{align*}
	\alpha_1 &= 2(c - \bar{A} - \bar{B}),\\
	\alpha_2 &= c^{-3}\left(c^4 - 2(\bar{A} +\bar{B})c^3 + (\bar{A}^2 + 4\bar{A}\bar{B} + \bar{B}^2)c^2
	 -2\bar{A}\bar{B}(\bar{A} +\bar{B})c + \bar{A}^2\bar{B}^2\right),\\
	\alpha_3 &= -2c^{-3}\left(c^4 -(\bar{A} + \bar{B})c^3 + \bar{A}\bar{B}(\bar{A} +\bar{B})c 
	- \bar{A}^2\bar{B}^2\right),\\
	\alpha_4 &= c^{-2}\bar{A}^2\bar{B}^2.
\end{align*}
Since $\varphi$ is single-valued, we must have that $\alpha_1 = 0$, i.e., $A + B = c$. This reduces the complexity
of the other $\alpha_i$'s:
\begin{align*}
	\alpha_2 &= \frac{\bar{A}^2\bar{B}^2}{(\bar{A}+ \bar{B})^3},\\
	\alpha_3 &= \frac{2\bar{A}\bar{B}(\bar{A}^2 + \bar{A}\bar{B} +\bar{B}^2)}{(\bar{A}+ \bar{B})^3},\\
	\alpha_4 &= \frac{\bar{A}^2\bar{B}^2}{(\bar{A}+ \bar{B})^2}.
\end{align*}
Since it is also required that $\alpha_3$ vanishes, one possibility (and the case that will be considered in the
next section) will be when $A^2 + AB + B^2 = 0$. Note that the only other choices are $A = 0$ or $B = 0$, which
both lead to the circle as a solution; we are not interested in this here. The condition $A^2 + AB + B^2 = 0$, 
along with the condition that $A + B = c$ implies that $A = ce^{\pi i/3}$, $B = ce^{-\pi i/3}$, and thus, 
setting $C = 1$, that
	\begin{equation} \label{conformalmap}
		\varphi(w) = w^{-1} + \frac{c}{1-cw} - c^2 w.
	\end{equation}
We now show that $\varphi(w)$ satisfies equation \eqref{dropleteq} with $p = 1 - c^2$, $\tau = 1 - c^2 + c^4$, and
some $F$ analytic in the complement of the droplet $\Omega$ except for a simple pole at infinity. This can be done
by writing $\bar{z} = S(z)$, the Schwarz function of the boundary of $\Omega$. In terms of the conformal map
$z = \varphi(w)$, the Schwarz function is
	\begin{equation}
		S(z) = \varphi(1/\varphi^{-1}(z)) = w - \frac{cw}{w-c} - \frac{c^2}{w},
	\end{equation}
where $w = \varphi^{-1}(z)$. Furthermore, the solution to equation \eqref{modF} can be determined, using equations
\eqref{mappedQD2}, \eqref{QD2-form}, and \eqref{conformalmap}:
	\begin{equation}
		F(z) = \frac{\sqrt{Q(w)}}{\varphi'(w)} = -i\frac{(w^2 -cw + c^2)(1-cw)}{(1-cw+c^2w^2)(w-c)}. 
	\end{equation}
Note that $F(z)$, $S(z)$ have simple poles at $z = 0, \infty$; the residues of these functions at $z = 0$ are:
	\begin{equation}
		\underset{z=0}{\text{Res }} S(z) = \frac{(1-c^2+c^4)^2}{(1-c^2)^2},
		\qquad \underset{z=0}{\text{Res }} F(z) = i\frac{1-c^2+c^4}{1-c^2}.
	\end{equation}
Therefore, the function $G_c(z)$, defined by
	\begin{equation}
		G_c(z) = (1-c^2) S(z) + i(1-c^2+c^4) F(z)
	\end{equation}
is analytic in the complement of the droplet, except for a simple pole at infinity. Additionally, because of how
$F(z)$, $S(z)$  were defined, on the boundary of the droplet, we have that
	\begin{equation}
		G_c(z) = p \bar{z} + i\tau \dot{\bar{z}},
	\end{equation}
where $p = 1 - c^2$, $\tau = 1 - c^2 + c^4$.
\end{proof}

\begin{remark}
    We remark that none of the droplets above are physical, as $\sqrt{G_c'(z)}$ cannot
    be defined in a single-valued way in the complement of the droplet. $\hat{G}_c(w)$ can be written as
    \begin{equation}
        \hat{G}_c(w) = \frac{p_4(w)}{w(c^2w^2-cw+1)},
    \end{equation}
    for some degree $4$ polynomial. By some rather involved computations, one finds that (in the local variable $z = \varphi^{-1}(w)$:)
    \begin{equation}
        G_c'(z) = \frac{2(c^2-1)(1-cw)^2(c^2/2w^4 + (c^4+1)w^2+c^2/2))}{(c^2w^2-cw+1)^2w^2}.
    \end{equation}
    (Here, $G'_c(z) = \hat{G}'_c(w)/\varphi'(w)$, with the $'$ on the right hand side denoting the derivative with respect to $w$).
    The above formula tells us that the only obstruction in defining $\sqrt{G_c'(z)} = \sqrt{\hat{G}'_c(w)/\varphi'(w)}$ in a single-valued
    way on the disc in the $w$-plane is if the polynomial $p(w) := (c^2/2w^4 + (c^4+1)w^2+c^2/2))$ has roots (of multiplicity 1) 
    in the disc. Indeed, the roots of $p(w)$ are
        \begin{equation}
            w_{\pm} = \frac{i}{c}\sqrt{1+c^4\pm\sqrt{c^8+c^4+1}},
        \end{equation}
    and $-w_{\pm}$. The roots $w_+, -w_+$ are always of modulus greater than 1; however, the roots $w_-, -w_-$ always lie in the disc, and never coincide; thus, $\sqrt{G_c'(z)}$ cannot be defined in a single-valued way in the complement of the droplet.
\end{remark}

\section{A particular 1-parameter family.}
Here, we will analyze the $1$-parameter family of mathematical droplets from the previous section, which are
characterized by the conformal map $\varphi_c(w)$ from the unit disc to the complement of the droplet, $\Omega$:
	\begin{equation} \label{droplet}
		\varphi_c(w) = \frac{1}{w} - \frac{c}{1-cw} - c^2 w.
	\end{equation}
In this section, we will prove the following proposition:
\begin{prop} \label{3.1}
The map $\varphi_c: \mathbb{D} \to \Omega$ is univalent for $0 \leq c < c^* := (\sqrt{5}-1)/2$; furthermore,
the boundary curve is convex only for $0 \leq c < c_1 = (3 - \sqrt{5})/2$.
\end{prop}
The content of this proposition is illustrated by Figure \ref{fig: shape of droplet evo}. The proof of this 
proposition, while straightforward, is rather tedious. For showing univalency of $\varphi_c$, the scheme of the 
proof follows that of most in the theory of univalent functions: show that the image of the boundary is a closed, 
Jordan curve, and conclude using argument principle that $\varphi_c$ is univalent. Convexity of $\varphi_c$ in the 
given parameter range is shown by a similar analysis of the boundary curve. 
We begin by analyzing the convexity of the boundary curve. 
We will utilize the following fact: a plane curve is convex if and only if its (signed) curvature has constant sign. We use the following lemma to help us derive an expression for the curvature in the
local variable $w$:
\begin{lemma}\label{curvature-lemma}
    Let $\gamma$ be a simple, closed curve, $\varphi: \mathbb{D} \to \Omega$, with $\Omega$ unbounded and $\partial \Omega =\gamma$, $z = \varphi(w)$. Then, in the variable $w$, the signed curvature
    of $\gamma$, $\kappa(\varphi(w)) =\hat{\kappa}(w)$, admits the expression
        \begin{equation}
        \hat{\kappa}(w) = 
            \frac{i}{2} \sqrt{ \frac{\varphi'(w)}{\hat{S}'(w)} } \left( \frac{\hat{S}''(w)}{\hat{S}'(w) \varphi'(w)} - 
            \frac{\varphi''(w)}{\varphi'(w)^2}\right),
        \end{equation}
    where $\hat{S}(w) := S(\varphi(w))$, and $S(z)$ is the Schwarz function of $\Omega$. All derivatives in the above formula are taken with respect to the local variable $w$.
\end{lemma}
\begin{proof}
    First, the signed curvature in terms of the arc length parameter $s$ on $\gamma$ can be written as
        \begin{equation}
            \kappa(s) = i\ddot{\bar{z}}/\dot{\bar{z}}.
        \end{equation}
    Furthermore, since we can write $\bar{z}$ in terms of the Schwarz function $\bar{z} = S(z)$ , we have the following 
    identities, using the chain rule and the identity $|\dot{z}| = 1$:
        \begin{equation}
            \dot{\bar{z}} = \sqrt{S'(z)}, \qquad \ddot{\bar{z}} = \frac{1}{2}\frac{S''(z)}{S'(z)} \qquad \Rightarrow
            \qquad \kappa(z) = \frac{i}{2}\frac{S''(z)}{S'(z)^{3/2}}.
        \end{equation}
    We would like to express the curvature in terms of the local variable $w = \varphi^{-1}(z)$; in this variable, we obtain the
    formulae:
        \begin{equation}
             S'(z) = \frac{\hat{S}'(w)}{\varphi'(w)}, \qquad S''(z) = \hat{S}''(w)(\varphi'(w))^{-2} -
             \hat{S}'(w) \frac{\varphi''(w)}{\varphi'(w)^3}.
        \end{equation}
    (By an abuse of notation, we denote derivatives with respect to $z$ and $w$ both by prime). Thus, in the local coordinate 
    $w$, the curvature is
        \begin{equation}
            \hat{\kappa}(w) = 
            \frac{i}{2} \sqrt{ \frac{\varphi'(w)}{\hat{S}'(w)} } \left( \frac{\hat{S}''(w)}{\hat{S}'(w) \varphi'(w)} - 
            \frac{\varphi''(w)}{\varphi'(w)^2}\right).
        \end{equation}
\end{proof}

Let $w = e^{i\theta}$, $-\pi \leq \theta \leq \pi$, and set
$\varphi(e^{i\theta}) := x(\theta) + i y(\theta)$.
It is clear from the definition that $x(\theta) = x(-\theta)$, $y(\theta) = -y(-\theta)$. Thus, the boundary curve is 
symmetric with respect to the $x$-axis, and we have only to check that the curvature does not change sign for
$0 \leq \theta \leq \pi$. Using Lemma \eqref{curvature-lemma}, our expression for the conformal map $\varphi(w)$, and the fact that $S(w) = \varphi(1/w)$ in the local coordinate $w$, we obtain the 
following expression for the curvature:
    \begin{equation}
        \hat{\kappa}(w) = \frac{w^3\left((2c^5-2c)w^2 + (1+4c^2-4c^4-c^6)w + (2c^5-2c) \right)}
        {(c^2w^2 -cw +1)^2(c^2-cw+w^2)^2},
    \end{equation}
or, in setting $w = e^{i\theta}$,
    \begin{equation} \label{curvature1}
        \hat{\kappa}(e^{i\theta}) = (c^2-1)\frac{(4c^3+4c)\cos \theta -c^4-5c^2-1}
        {(4c^2\cos^2 \theta - (2c^3+2c)\cos \theta +c^4-c^2 +1)^2}.
    \end{equation}
The denominator of \eqref{curvature1} clearly has constant sign; thus, it suffices to 
check that the numerator also has constant sign. Let $t = \cos\theta$, and consider the polynomial 
    \begin{equation}
        p(t) = (4c^3+4c)t -c^4-5c^2-1.
    \end{equation}
We must check that $p(t)$ has no zeros in $[-1,1]$, so that $p(\cos\theta)$ maintains constant
sign for $0 \leq \theta \leq \pi$. $p(t)$ is linear, and thus monotone; therefore, we must
find the solutions to $p(1) = 0$, i.e., the first value of $c$ for which $p$ vanishes in $[-1,1]$. The equation $p(1) = 0$ yields
    \begin{equation*}
        c^4 - 4c^3 + 5c^2-3c+1 = (c^2-3c+1)(c^2-c+1);
    \end{equation*}
the solution lying in the interval $[-1,1]$ is $(3-\sqrt{5})/2 = 0.38197...$. This shows that the droplet is convex in the range specified in Proposition \ref{3.1}.

\begin{prop}
    For $0 \leq c < c^* = \frac{\sqrt{5}-1}{2}$, $\varphi_c(w)$ is univalent.
\end{prop}
\begin{proof}
    Note that the vertical line $\Re{w}=x(0)=\frac{-c^3+c^2+2c-1}{c-1}$ intersects $\Gamma(c):=\varphi_c(\partial \mathbb{D})$ at points $\theta$ (besides the trivial case $\theta=0$) such that: $$\cos{(\theta)}= \frac{c^4-2c^3+c^2-2c+1}{2c(c^2-2c+1)}=:\alpha(c).$$ 
    It is easy to check that $\alpha(c)$ is monotonically decreasing for $c\in [0,c^*)$. Moreover, $\alpha(c^*)=-1$ and $\alpha(c_1)=1$. Thus the evolution of the droplet with respect to $c$ can be divided into 3 stages: (I) $0 \leq c < c_1$; (II) $c_1<c<c^*$; (III) $c^*<c<1$.
    Let us denote the vertical lines $\Re{w}=a$ by $v_a$. We remark that univalency of $\varphi$ when $c$ is in stage \textit{I} is established by the argument principle, and
    convexity of the curves for this parameter range.
    
    Thus, we begin with stage II. In stage II, consider the equation
    \begin{align*}
        x(\theta)= \frac{2c(c^2-1)\cos^2\theta + (1+c^2-c^4) \cos\theta -c}{c^2 -2c \cos \theta + 1} = a,
    \end{align*}
    which can be regarded as finding the zeros of the following parabolic function of $\cos{(\theta)}$:
    \begin{align}
        f(\cos(\theta))=(2c^3-2c)\cos^2(\theta)+(-c^4+c^2+2ac+1)\cos(\theta)-c-a(1+c^2).
        \label{parabolic function}
    \end{align}
    A simple calculation gives us the discriminant of $f$:
    \begin{align}
        \Delta=4c^2(a-a_1)(a-a_2),
    \end{align}
    where 
    \begin{align*}
        a_1&=-\frac{c^4+2c^3+c^2-2c-1}{2c},\\
        a_2&=-\frac{c^4-2c^3+c^2+2c-1}{2c}.
    \end{align*}
    It is also easy to show that $a_2>a_1>x(0)$ for $c\in (0,1)$. Since $2c^3-2c$ is strictly negative and $-\frac{c^2-c^4+1+2ac}{4(c^2-1)c}\in (0,1)$, to show that $f(\cos\theta)$ has two zeros in $(0,1)$, it is sufficient to show that $f(1)<0$ and $\Delta>0$. In fact, for $x(0)<a<a_2$, $f(1)<0$, thus $v_a$ intersects $\Gamma(c)$ at 4 points which are symmetric with respect to $x-$axis. Now for $a\in (x(\pi),x(0))$, one can show $v_a$ intersects $\Gamma(c)$ at 2 points, by a similar analysis on $\Delta$. A typical case in stage II is as shown in Figure \ref{fig: shape of droplet evo}, the case of $c=0.5$.
    
    In the stage III, by a similar argument, one can show that for $a<x(0)$, $v_a$ does not intersect $\Gamma(c)$ and for $a\in (x(0),x(\pi))$, $v_a$ intersects $\Gamma(c)$ at 2 points. For $a\in (x(\pi),a_*)$, $\Delta>0$ and $f(\pm 1)<0$ imply that $f$ has 2 solutions hence the intersection points are 4. Here $a^*$ is the value such that $f(-1)$ vanishes. And, for $a\in (a^*,a_1)$, since $\Delta>0$ and $f(\pm 1)<0$, again, we have 4 intersection points. After $a>a_1$, $\Delta<0$, there is no solution, and hence no intersection.
    \begin{figure}[ht]
        \centering
        \includegraphics[width=0.7\textwidth]{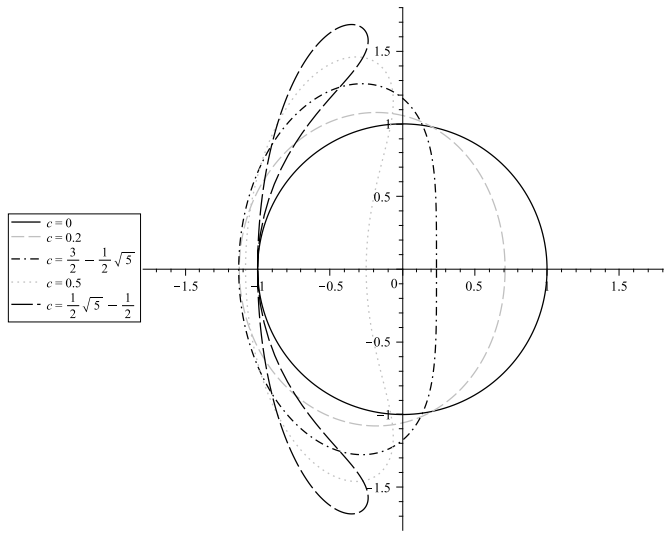}
        \caption{Evolution of the shape of the droplet with respect to $c$.}
        \label{fig: shape of droplet evo}
    \end{figure}
    To summarize the above, we conclude that when $0\leq c < c^*$, the map $\varphi_c(w)$ is univalent.
\end{proof}

\section{A generalization of the droplet in \S2.}
In \S2, we solved the following problem: Suppose $\gamma$ is a rectifiable Jordan curve, with exterior $\Omega$ 
being a Smirnov domain. Furthermore, let $F$ be in class $E^1$ near the boundary of the droplet, with
	\begin{equation} \label{modF2}
		F(z) = \dot{\bar{z}} \qquad\text{a.e. on $\gamma$.}
	\end{equation}
Finally, suppose that $F$ has a single pole at a finite point $c \in \Omega$, and $F(\infty) \neq 0$.
In \S2 (and with the analysis of \S3), we were able to determine a 1-parameter family of mathematical droplets,
which solved equation \eqref{dropleteq} with a particular choice of $F$, $p$, and $\tau$, using the solution to
the equation \eqref{modF2}.

With the purpose of finding more mathematical droplets, we generalize the above setting, by trying to find 
solutions to \eqref{modF2} with the modification that $F$ is now allowed to have multiple poles in the complement
to the droplet. The idea here is the following: allowing additional poles of $F$ in the complement to the droplet will
in turn result in additional parameters in the solution to the problem, i.e., in the conformal map $\varphi$ from the
disc to the complement of the droplet, the Schwarz function $S(z)$, and in the function $F$ itself. 
Additionally, if we want the function
	\begin{equation}
		G(z) := p S(z) + i \tau F(z)
	\end{equation}
to have only a single pole at infinity, similarly to the analysis in \S2, we must have that the residues of the 
poles of $p S(z)$ at the other finite points cancel the residues of $i \tau F(z)$; the additional parameters that
appear in the problem will give us enough freedom to cancel these residues. 

Indeed, it is possible to do this, following the procedure outlined above. The resulting family of droplets is 
depicted in Figure \ref{fig:NewDroplet}. With the philosophy of the above method in mind, we state the following theorem:
    \begin{theorem}
        Let $\varphi_{c} : \mathbb{D} \to \Omega_{c}$ be the family of univalent maps defined by the formula
            \begin{equation}
                \varphi_{c}(w) = -\frac{1}{w} + \frac{4c}{1-cw} - \frac{4c}{1+cw},
             \end{equation}
        where $0 < c < \frac{1}{3}$. Then, the droplet whose complement is $\Omega_{c}$ satisfies the equation
        \begin{equation}
            G_{c}(z) = p \bar{z} + i \tau \dot{\bar{z}}.
        \end{equation}
        on the boundary of the droplet, with $p = 1-c^4$, $\tau = 6c^4 + 2$, 
        and where $G_{c}(z)$ is defined by the formula
        \begin{equation}
             G_c(\varphi(w)) := \hat{G}_c(w) = (9c^4-1)\frac{w(3+c^2w^2)}{1+3c^2w^2}
        \end{equation}
        in the local variable $w = \varphi_{c}^{-1}(z)$. Furthermore, the droplet is convex for 
        $0 < c \leq c_1 = \frac{\sqrt{6\sqrt{13}-21}}{3} = 0.265...$
    \end{theorem}
\begin{figure}[ht]
    \centering
    \includegraphics[width=0.6\textwidth]{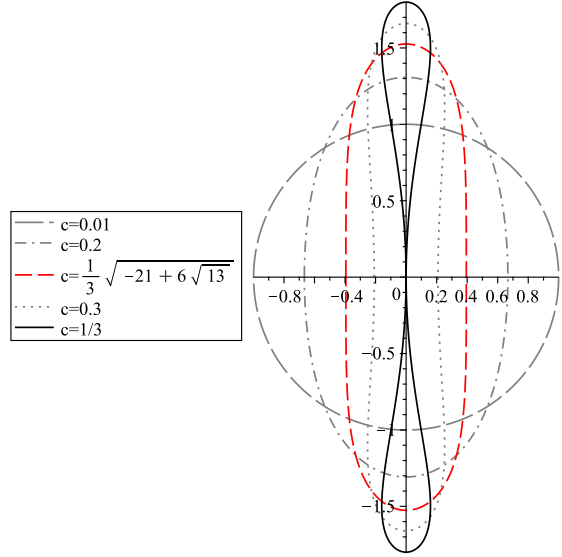}
    \caption{The family of physical droplets constructed in \S4. The droplets are similar in shape to the 1-parameter
    family in \cite{KSV} which contain the McLeod droplet; however, here, all droplets are physical.}
    \label{fig:NewDroplet}
\end{figure}
\begin{proof}
We now begin our computation; the convexity of the family of curves, as well as the univalency of $\varphi$ in the given
parameter range will be assumed here; proof of these facts will be postponed until \S5. 
Suppose $F$ has two poles in the complement of the droplet: without loss of generality, at $0,z_0 \in \Omega$,  and 
suppose $F(\infty) \neq 0$. Let $\varphi$ be the conformal map to the complement of the droplet, normalized so that 
$\varphi(0) = \infty$, $\varphi(c) = 0$, for some $0 < c < 1$. We have no freedom in choosing the location of the point 
which is mapped to $z_0$; we let $|q| < 1$ so that $\varphi(q) = z_0$. This conformal map is illustrated in Figure 
\ref{fig:ModifiedDroplet}. 
\begin{figure}[ht]
    \centering
    \includegraphics[width=.8\textwidth]{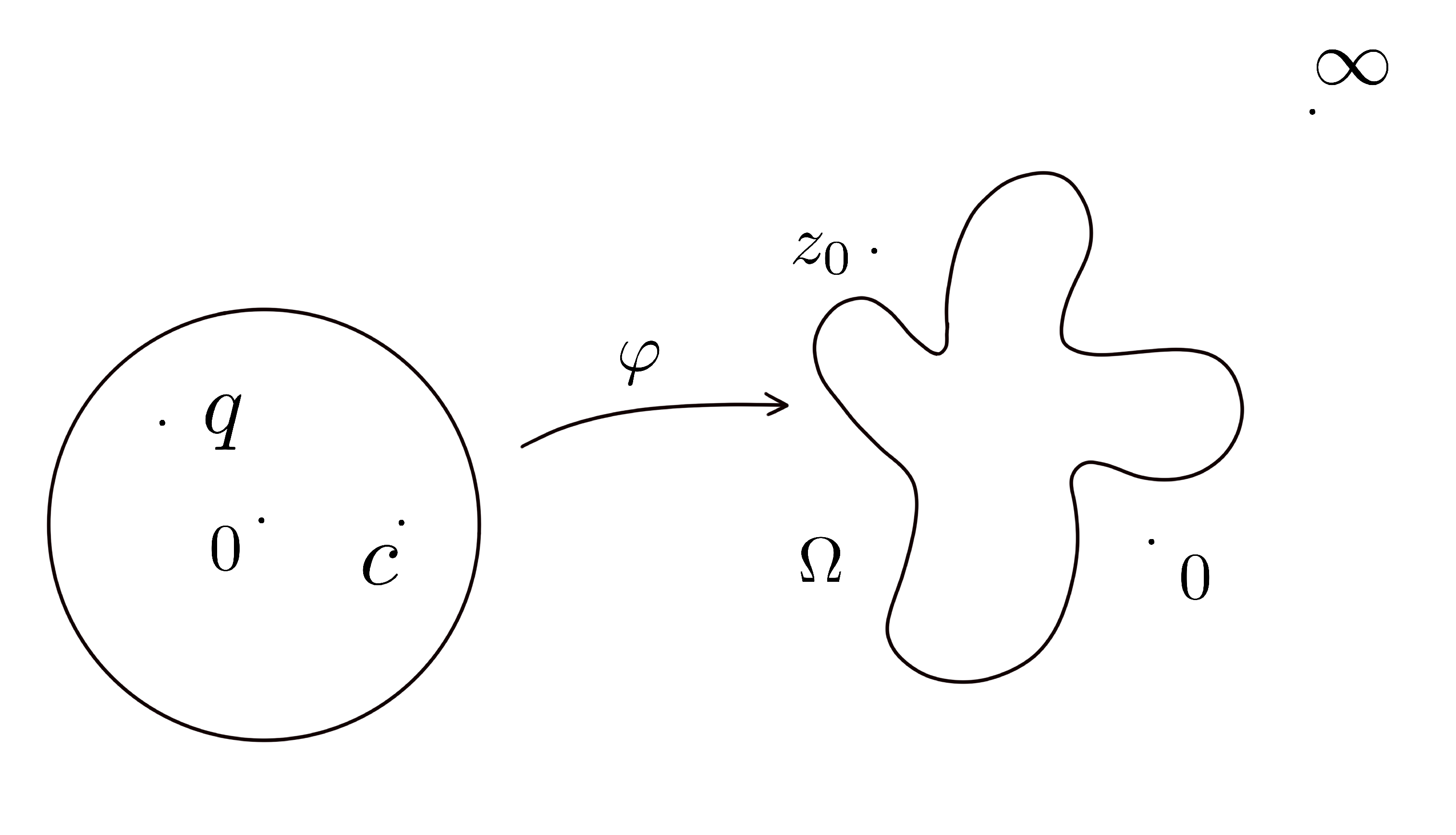}
    \caption{The conformal map $\varphi(w)$ from the disc to the 
    complement of the droplet. $\varphi$ is normalized so that $\varphi(0) = \infty$, $\varphi(c) = 0$, for some 
    $0 < c < 1$. The location of $q: \varphi(q) = z_0$ is fixed by these conditions.}
    \label{fig:ModifiedDroplet}
\end{figure}
Consider the 
quadratic differential 
	\begin{equation} \label{newQD}
		Q := Q(w) dw^2 = F^2(\varphi(w)) (\varphi'(w))^2 dw^2.
	\end{equation}
This quadratic differential has the following properties:
	\begin{enumerate}
		\item $Q$ has poles of order $2$ at $w = c$ and $w = q$.
		\item $Q$ has a pole of order $4$ at $w = 0$.
		\item $Q$ is positive on the unit circle.
	\end{enumerate}
The last property (along with the reflection principle) allow us to extend $Q$ to a quadratic differential on the
entire sphere; thus (again, see Appendix \ref{A}), $Q$ is necessarily of the form
	\begin{equation} \label{form-newQD}
		Q = -C^2\frac{(w-A)^2(w-B)^2(w-E)^2(1-\bar{A}w)^2(1-\bar{B}w)^2(1-\bar{E}w)^2}
		{w^4(w-c)^2(1-cw)^2(w-q)^2(1-\bar{q}w)^2}dw^2,
	\end{equation}
for some $|A|,|B|,|E| < 1$, $C^2 >0$; since $C$ will eventually only contribute a global dilation of the droplet, to simplify matters, we
set $C = 1$. Equations \eqref{modF2}, \eqref{newQD}, and \eqref{form-newQD} imply that 
	\begin{equation}
		|w^2 \varphi'(w)| = \frac{|(1-\bar{A}w)(1-\bar{B}w)(1-\bar{E}w)|^2}{|(1-cw)(1-\bar{q}w)|^2} 
		\qquad \text{a.e. $|w| = 1$}.
	\end{equation}
Since the disc is a Smirnov domain, we can recover $\varphi'(w)$ from its $a.e.$ boundary values. We find that
	\begin{equation} \label{newphi_p}
		\varphi'(w) = \frac{(1-\bar{A}w)^2(1-\bar{B}w)^2(1-\bar{E}w)^2}{(1-cw)^2(1-\bar{q}w)^2w^2},
	\end{equation}
for $|w| < 1$. At this stage, the computation becomes rather involved, due to the number of terms we must
 integrate. We used Maple to compute the integral of $\varphi'(w)$. Upon integration, we obtain terms of the form
 $\log(1-\bar{q}w)$, $\log(1-\bar{c}w)$, $\log(w)$; in order for $\varphi$ to be single valued, the coefficients 
 of these terms must vanish. This puts constraints on the parameters $A, B, E, c,$ and $q$. Using Maple, we 
 eliminate the parameters $A,B,$ and $E$ to obtain an expression for $\varphi := \varphi_{c,q}$. There are several 
 solutions to this system of equations, but all either reduce to a previous case (i.e., the droplet from \S2 or the
 circle), or, are equivalent to the following solution up to a relabelling of $A$, $B$, and $E$:
    \begin{equation*}
        A = e^{\pi i/3}c + e^{-\pi i/3}q, \qquad B = 0, \qquad E = \bar{A};
    \end{equation*}
we additionally determine that $q$ must be real, which we state here, to make the formulas that follow more
manageable. This in fact follows from the condition that the residues of the Schwarz function of the domain and the
residues of $F$ must cancel, and so these functions must have poles at the same locations; this gives us the constraint 
$q = \bar{q}$, i.e., $q$ must be real. 

With these computations in place, we find that the conformal map $\varphi := \varphi_{c,q}$ is:
	\begin{equation} \label{new_phi}
		\varphi_{c,q}(w) = -\frac{1}{w} + \frac{(c-q)^2}{q(1-qw)} + \frac{(c-q)^2}{c(1-cw)}.
	\end{equation}
(Here, as in previous sections, we drop the integration constant; this breaks the condition that $\varphi(c) = 0$, but will ease the exposition). The Schwarz function, in the variable $w = \varphi^{-1}(z)$, is then
	\begin{equation}
		S(\varphi(w)) := \hat{S}(w) = -w + \frac{c^3-c^2q-cq^2 + q^3}{cq} + \frac{(c-q)^2}{w-c} +  \frac{(c-q)^2}{w-q},
	\end{equation}
and, using equations \eqref{newQD}, \eqref{newphi_p},
	\begin{equation}
		F(\varphi(w)) := \hat{F}(w) = \frac{\sqrt{Q}}{\varphi'} = i \frac{w(1-qw)(1-cw)(w-A) 
		(w-\bar{A})}{(q-w)(c-w)(1-\bar{A}w)(1-Aw)}.
	\end{equation}
Now, we must find constants $p$, $\tau$ so that
	\begin{align*}
		\underset{w=c}{\text{Res }} p S(z) &= -\underset{w=c}{\text{Res }} i\tau F(z);\\
		\underset{w=q}{\text{Res }} p S(z) &= -\underset{w=q}{\text{Res }} i\tau F(z).\\
	\end{align*}
These equations give us additional constraints on the parameters involved in the problem; again using Maple, we find 
that $c$ and $q$ must be algebraically related. There are three solutions:
    \begin{enumerate}
        \item $c = q$, in which case it follows from \eqref{new_phi} that the droplet is a circle;
        \item $c = -q$, which results in a $1$-parameter family of droplets, with $p = 1-c^4$, $\tau = 6c^4 + 2$.
        \item $(c,q)$ lie on the ellipse $c^2 + q^2 -cq -1 = 0$, in which case, after some computation, 
        the solution is again a circle.
    \end{enumerate}
Thus, the only case of interest is case 2. In the second case, we can simplify $\varphi(w)$ as:
    \begin{equation}
       \varphi_{c}(w) = -\frac{1}{w} + \frac{4c}{1-cw} - \frac{4c}{1+cw};
    \end{equation}
this map is univalent, provided $0 < c < \frac{1}{3}$, as will be shown in \S5.
We thus see that the droplet satisfies the equation
    \begin{equation}
        G_{c}(z) := p S(z) + i \tau F(z)
    \end{equation}
where $G_{c}(z)$ is given in the variable $w = \varphi^{-1}(z)$ by
    \begin{equation}
         G_{c}(z) = (9c^4-1)\frac{w(3+c^2w^2)}{1+3c^2w^2},
    \end{equation}
and the Schwarz function $S(z)$ and the function $F(z)$ have the form (again in the variable $w = \varphi^{-1}(z)$)
    \begin{equation}
        S(z) = \frac{w(9c^2- w^2)}{w^2-c^2}, \qquad F(z) = i\frac{w(w^2+3c^2)(1-c^2w^2)}{(1+3c^2w^2)(w^2-c^2)}.
    \end{equation}
Finally, we can check if the droplets are physical by seeing if $\sqrt{G_{c}'(w)}$ is single-valued in the disc (in the $w$-plane).
We find that
    \begin{equation}
        \sqrt{G_{c}'(z)} = \sqrt{\frac{(27c^4-3)(1-c^2w^2)^2}{(3c^2w^2 + 1)^2}} = \frac{\sqrt{27c^4-3}(1-c^2w^2)}{(3c^2w^2 + 1)},
    \end{equation}
so that $\sqrt{G_{c}'(z)}$ can be defined in a single-valued way in the disc. Therefore, since $\varphi_c$ is univalent 
for $0 < c < \frac{1}{3}$, all droplets in the family are physical. 
\end{proof}
\begin{remark} 
    The formula \eqref{new_phi} contains the previous family of droplets as a special case. Clearly, when $c = q$, 
    the formula \eqref{new_phi} reduces to $\varphi_{c,q}(w) = \frac{1}{w}$, i.e., the droplet is again a circle.
    Let us shift the conformal map $\varphi_{c,q}(w)$ by the constant $-\frac{c^2}{q} + 2c$; denote this new map by 
    $f_{c,q}(z) := \varphi_{c,q}(w) -\frac{c^2}{q} + 2c$. Expanding the second term of \eqref{new_phi} in powers of 
    $q$, we find that
        \begin{align*}
            f_{c,q}(z)&= -\frac{1}{w} + \left( \frac{c^2}{q}-2c+q + (c^2-2cq+q^2)w + \cdots \right) + \frac{(c-q)^2}{c(1-cw)} 
            -\frac{c^2}{q} + 2c\\
            &= -\frac{1}{w} + \left(c^2 w + \OO(q) \right) + \frac{(c-q)^2}{c(1-cw)},
        \end{align*}
    and, in the limit as $q \to 0$, we recover the map \eqref{conformalmap}, albeit with an overall minus sign (which corresponds to a global rotation of the droplet).
\end{remark}

\begin{remark}
    We summarize the behavior of the family of droplets; cf. Figure\eqref{fig:NewDroplet}:
        \begin{itemize}
            \item As $c \to 0$, the droplets approach the circle.
            \item For $ 0 < c \leq c_1 = \frac{\sqrt{6\sqrt{13}-21}}{3} = 0.265...$, the droplets are convex; precisely at $c = c_1$,
            the droplets cease to be convex.
            \item For $ c_1 < c < c^* = 1/3$, the droplets are no longer convex, and begin to `collapse' along the $x$-axis; when 
            $c = 1/3$, the boundary of the droplet degenerates at $0$. The droplet pinches off into two connected components. The behaviour of the droplet after this critical point would require a more sophisticated analysis.
        \end{itemize}
\end{remark}


The 1-parameter family of physical droplets found above has in fact been discovered before by D. Crowdy \cite{DC} in the context of
a problem in vortex dynamics. The problem considered there shares the same mathematical model as the electrified droplet considered
here, as demonstrated explicitly by D. Crowdy \cite{DC-15}. Thus, solutions to one problem are simultaneously solutions to the other,
albeit with rather distant physical interpretations. The latter problem is in fact the one considered originally by McLeod \cite{M};
the problem we consider is the one considered first by Garabedian \cite{G}. We have therefore demonstrated that our method is 
capable of producing the same solutions as the ones previously discovered by Crowdy.

The solutions found above have 2 poles, and exhibit a 2-fold rotational symmetry. By a clever symmetry argument, 
the above solution can by generalized to an $m$-pole solution with an $m$-fold rotational symmetry. This result
is due originally to R. Wegmann and D. Crowdy \cite{WC}, again in the context of vortex dynamics. The essential idea is to 
replace the conformal map $\varphi(w)$ found above, given by
    \begin{equation}
        \varphi(w) = \frac{1}{w} + \frac{8c^2w}{c^2w^2-1},
    \end{equation}
by the more general ansatz
    \begin{equation}
        \varphi(w):= \varphi_m(w)= \frac{1}{w} + \frac{Aw^{m-1}}{Bw^m+C},
    \end{equation} 
for some to be determined coefficients $A,B,C$. If a solution of this form exists, it will have an $m$-fold rotational symmetry. 
Indeed, working through the computation explicitly, one finds that 1-parameter families of solutions exist for each $m$. As
previously mentioned, these solutions have already been found by Wegmann and Crowdy, and so we only state the result here in the 
terminology of this paper. The resulting droplets have conformal maps given by
    \begin{equation}
        \varphi_m(w) = \frac{1}{w} + \frac{4m(icw)^m}{w(m-1)(1+(m-1)(icw)^m)}.
    \end{equation}
These maps are univalent, provided the parameter $c$ is sufficiently small; by a similar analysis to that which will appear in \S5, 
one finds that $\varphi_2(w)$ is conformal for $0 < c < 1/3 = 0.33333...$, $\varphi_3(w)$ is conformal for 
$0 < c < \left(\frac{\sqrt{2}-1}{4}\right)^{1/3} = 0.46959...$, $\varphi_4(w)$ is conformal for 
$0 < c < \left(\frac{37-8\sqrt{10}}{135}\right)^{1/4} = 0.54259...$, and so forth. It would be interesting 
to find a general formula for the range of univalency of the maps $\varphi_m(w)$, for arbitrary $m$. 
We also list all of the relevant formulae from the previous computations, i.e. the Schwarz function $\hat{S}_m(w)$,
$\hat{F}_m(w)$, $\hat{G}_m(w)$, and $p_m$, $\tau_m$.

In terms of the conformal map $\varphi_m(w)$, we have that, for $|w| = 1$, $\hat{S}_m(w) = \overline{\varphi_m(w)}$. In the complement of
the droplet, this continues analytically to the function
    \begin{equation}
        \hat{S}_m(w) = \frac{w\big[(m-1)w^m+(ic)^m(m+1)^2\big]}{(m-1)\big[w^m+(m-1)(ic)^m\big]}.
    \end{equation}
In terms of the conformal map $\varphi_m(w)$ and the Schwarz function $\hat{S}_m(w)$, we have that 
$\hat{F}_m(w) = \sqrt{\hat{S}'_m(w)/\varphi'_m(w)}$. This function continues analytically to the complement of the droplet as 
    \begin{equation}
    \hat{F}_m(w) = i\frac{w\big[(m-1)(icw)^m+1\big]\big[w^m-(m+1)(ic)^m\big]}{\big[(m+1)(icw)^m-1\big]\big[w^m+(m-1)(ic)^m\big]}.
    \end{equation}
The pressure and surface tension proportionality constants are then determined up to an overall scale factor, so as to guarantee the 
combination $\hat{G}_m(w):= p_m\hat{S}_m(w) + i \tau_m \hat{F}_m(w)$ has no poles in the disc. We find that
    \begin{equation}
        p_m = (m-1)(1 - (m-1)^2c^{2m}), \qquad \tau_m = (2m^2-2)c^{2m}+2.
    \end{equation}
The general formula for $\hat{G}_m(w)$ is, up to an overall numerical constant $A_m$ (depending on $c$,$m$):
    \begin{equation}
        \hat{G}_m(w) = -A_m(c) w\frac{m+1-(m-1)^2(icw)^m }{(m+1)(icw)^m-1}.
    \end{equation}
Differentiating with respect to $w$, we find that $\hat{G}'_m(w)$ is a perfect square (up to a numerical factor):
    \begin{equation}
        \hat{G}_m'(w) = -A_m(c)(m+1)\left(\frac{(m-1)(icw)^m+1}{(m-1)(icw)^m-1}\right)^2.
    \end{equation}
This formula is valid for any value of $c$, and so its square root is single-valued, and the solutions calculated here are indeed physical.

\begin{remark}
    As first noted by D. Crowdy and J. Roenby \cite{CR}, all solutions to the droplet equation discovered so far are examples of so-called \textit{double quadrature domains}; that is, domains which satisfy quadrature identities with respect to both area
    and arc length \footnote{For a more complete definition and general properties of double quadrature domains, cf. \cite{BGS},
    or Chapter 5 in \cite{Shapiro}.}. The authors suspect that generic solutions to the free boundary problem considered here must
    be of this type, but there is (as of the time of the writing of this work) no clear reason why there couldn't be solutions 
    which are not double quadrature domains. It is also not immediately clear what sorts of conclusions can be drawn from this 
    fact; nevertheless, it is certainly an interesting property that all solutions found thus far have in common.
\end{remark}

\begin{remark}
    It has also been pointed out that, very recently, the solutions described in this section also solve another problem in vortex 
    dynamics, that of rotating hollow vortices \cite{DNK-21}; these solutions have been dubbed ``H-states''. In this problem, there is no surface tension. The exact extent to which 
    these two problems are related is still an active area of research; the mathematical model used in \cite{DNK-21} (in particular 
    Formula 2.8) is almost exactly Formula \eqref{dropleteq} in the present work.
\end{remark}

Of course, the above methods constrain us to the case of a connected droplet; droplets with multiple connected components 
require different techniques, as the Riemann mapping theorem no longer applies. This will be discussed briefly in \S6.

\section{Analysis of the droplets in \S4.}
In this section, we analyze the convexity and univalency of the droplets from \S4.

\begin{prop}
    The droplets defined by the conformal map $\varphi_{c}(w) = -\frac{1}{w} + \frac{4c}{1-cw} - \frac{4c}{1+cw}$ are
    convex for $0 < c \leq c_1 := \frac{\sqrt{6\sqrt{13}-21}}{3} = 0.265...$.
    \label{prop: new drop convexity}
\end{prop}
\begin{proof}
    We use the following again use the geometric fact that a plane curve is convex if and only if its signed curvature has the same
    sign at every point.  
    Using Lemma \eqref{curvature-lemma}, our expression for the conformal map $\varphi(w)$, and the fact that $\hat{S}(w) = \varphi(1/w)$, 
    we find that $\kappa(\varphi^{-1}(z)) = \kappa(w)$ is given by
        \begin{equation}
            \hat{\kappa}(w) = \frac{3c^4 w^8 + (6c^2-18c^6)w^6 +(-9c^8+28c^4-1)w^4 + (6c^2-18c^6)w^2 + 3c^4}{(3c^2+w^2)^2(1+3c^2w^2)^2};
        \end{equation}
    or, substituting $w = e^{i\theta}$,
    \begin{equation}\label{curvature-new}
        \hat{\kappa}(e^{i\theta}) = \frac{48c^4 \cos^4\theta + (24c^2-48c^4-72c^6)\cos^2\theta -9c^8+36c^6+34c^4-12c^2-1}{(12c^2\cos^2 \theta + 9c^4-6c^2+1)^2}.
    \end{equation}
    We must show that this expression does not change sign for $0 \leq \theta \leq \pi$ (again, since the curve is
    symmetric, we only need to check the upper half). Clearly, the denominator of \eqref{curvature-new} is always positive; 
    thus, we must only check that the numerator doesn't change sign. Letting $t = \cos\theta$, we must equivalently check that
    the polynomial 
        \begin{equation}
            p(t) := 48c^4 t^4 + (24c^2-48c^4-72c^6)t^2 -9c^8+36c^6+34c^4-12c^2-1
        \end{equation}
    has no roots in $[-1,1]$, so that the curvature is of constant sign, and the curve is convex. To show $p(t)$ has no zeros in $[-1,1]$, it suffices to show that $p(0)< 0$ and $p(1)< 0$ due to the fact that $p(t)$ is even and the leading coefficient is positive. In fact, we have:
    \begin{align*}
        p(0)&=-(3c^2 - 1)(c^2 + 1)(3c^4 - 14c^2 - 1),\\
        p(1)&=-(c - 1)(c + 1)(3c^2 + 1)(3c^4 + 14c^2 - 1).
    \end{align*}
    Solving the inequalities $p(0)< 0, p(1)< 0$, we have
    \begin{align*}
        0 < c < c_1 = \frac{\sqrt{6\sqrt{13}-21}}{3}.
    \end{align*}
    
\end{proof}
\begin{prop}
    For $0<c<1/3$, the conformal map $\varphi_c(w)$ is univalent.
\end{prop}
\begin{proof}
   From Proposition \eqref{prop: new drop convexity}, we know that $\varphi_c(w)$ is univalent for $0 < c < c_1$.
   Next, fix $c\in (c_1,1/3)$, denote $x(\theta)=\Re{\varphi_c(w)}$. Consider equation 
   \begin{align}
    \label{new univalent}
       x(\theta) = a, \qquad a\in\RR.
   \end{align}
   We need to show that for $a\in (0,x(\pi))$, the equation has only two solutions, while for $a\in (x(\pi),a^*)$, the equation has four solutions, where $a^*$ is to be determined later. By taking the numerator of equation \eqref{new univalent}, we introduce a new function
   \begin{align}
       f(t)=-4c^2t^3 - 4ac^2t^2 + (9c^4 - 6c^2 + 1)t + ac^4 + 2ac^2 + a,
   \end{align}
   where $t=\cos(\theta)\in (-1,1)$. Now it is sufficient to show that for $c\in (c_1,1/3)$ fixed, $f$ has two solutions for $a\in (x(\pi),a^*)$, where $$x(\pi)=\frac{9c^2-1}{c^2-1}.$$
   Note the facts that
   \begin{align}
       f(0)&=a(c^2+1)^2>0,\\
       f(\pm 1)&=(c^2-1)(a(c^2-1)\pm (9c^2-1))<0,
   \end{align}
   provided that $a>\frac{9c^2-1}{c^2-1}$. Also, since the leading coefficient of $f$ is negative, we conclude that $f$ has exactly two solutions in $(-1,1)$ with different signs. Similarly, for $ a\in (0,x(\pi))$, one can show $\pm f(\mp 1)>0$, which implies that $f(t)$ has exactly one solution in $(0,1)$. Now, let us determine $a^*$. This comes from the discriminant of cubic equation: using Maple, it is easy to get that
   \begin{align}
       a^*=\frac{-4c + \sqrt{-27c^8 + 18c^4 + 8c^2 + 1}}{2(c^2 + 1)c}.
   \end{align}
\end{proof}
\begin{remark}
    $2a^*$ is actually the maximum width of the droplets.
\end{remark}

 \section{Multi-component droplets.}
 We now discuss the construction of multi-component droplets, i.e. droplets whose complement is a domain of connectivity $g \geq 1$. 
 Our main focus will be the two component case, and most of this
 section is dedicated to the analysis of this particular case. At the present time, the authors were unable to derive any exact solutions in this case; however, it is possible to write down an expression for the form of the conformal map for a particular solution, and so this is presented. At the end of this section, we will also discuss our speculations on how
 to explicitly construct solutions with more than two components.
 Throughout the following, let $r < 1$, and denote $\mathbb{A} =\{z \mid r < |z| < 1\}$ to be the annulus, and $\Gamma = \partial \mathbb{A}$
 its boundary. Our first objective is to establish results about quadratic differentials on the annulus, similar to those obtained for
 the disc in Appendix \ref{A}. We achieve this through the following lemma and proposition. Both are well-known results, and we present them mainly for the convenience of the reader.
 \begin{lemma}
 Suppose $f$ is meromorphic in $\mathbb{A}$ and continuous on $\bar{\mathbb{A}}$, with no poles on the boundary,  non-vanishing in $\Omega$, and $f|_{\Gamma} > 0$. Then, $f$ is a constant.
 \end{lemma}
 The proof of this lemma follows from elementary complex analysis, and so we omit it here.
 \begin{prop}
     Let $A, B \in \mathbb{A}$, such that $A/B \in \RR$. Then, there exists a unique (up to a constant multiple) 
     function $f_{AB}(z)$, meromorphic in $\mathbb{A}$, and positive on the boundary, such that $f_{AB}$ has a zero at $A$ and 
     a pole at $B$.
     \end{prop}
 \begin{proof}
     Uniqueness of $f_{AB}(z)$ follows immediately from the previous lemma. To show that such an $f$ exists, we construct it 
     explicitly. We claim that, up to a constant multiple, 
         \begin{equation}
             f_{AB}(z) = \frac{P(z/A) P(\bar{A}z)}{P(z/B) P(\bar{B}z)},
         \end{equation}
     where
     \begin{equation}
         P(z) = (1-z)\prod_{k=1}^{\infty}(1-r^{2k}z)(1-r^{2k}/z).
     \end{equation}
     The form of this solution can be inferred from the Schwarz reflection principle: since $f_{AB}$ has a zero (resp. pole) at $A$
     (resp. B), and is positive on $\Gamma$, by the reflection principle, $f$ must have zeros (resp. poles) at $Ar^{2n}$ and 
     $r^{2n}/A$ (and the corresponding places for $B$), $n \geq 1$. This function obviously satisfies these conditions, and its only 
     zero/pole in $\mathbb{A}$ is $A$ (resp. $B$). Therefore, it remains only to check that $f|_{\Gamma} > 0$. It is easily verified that
     the function $P(z)$ admits the following properties:
     \begin{equation}
         P(r^2z)  = -\frac{1}{z}P(z) = P(1/z).
     \end{equation}
     Using this identity, as well as the fact that, for $|z| = r$, $\bar{z} = r^2/z$, we see that
     \begin{equation}
         f_{AB}(z) = \frac{P(z/A) P(r^2/\bar{A}z)}{P(z/B) P(r^2/\bar{B}z)} =
          \frac{P(z/A) P(\bar{z}/\bar{A})}{P(z/B) P(\bar{z}/\bar{B})} = \frac{|P(z/A)|^2}{|P(z/B)|^2} > 0,
     \end{equation}
     i.e., $f_{AB} > 0$ on the circle $|z| = r$. 
     Now, on the circle $|z| = 1$, since $\bar{z} = 1/z$, we have that
     \begin{equation}
         f_{AB}(z) = \frac{-\bar{A}z P(1/\bar{A}z) P(z/A)}{-\bar{B}z P(1/\bar{B}z) P(z/B)}
             = \frac{\bar{A}}{\bar{B}} \frac{P(\bar{z}/\bar{A}) P(z/A)}{P(\bar{z}/\bar{B}) P(z/B)} = 
             \frac{\bar{A}}{\bar{B}} \frac{|P(z/A)|^2}{|P(z/B)|^2} > 0,
     \end{equation}
     since by assumption, $\bar{A}/\bar{B}$ is real (and, without loss of generality, is positive). This completes the proof.
 \end{proof}
 \begin{remark}
     The function $f_{AB}$ written above is an explicit construction of the map from the annulus to a parallel slit 
     domain. It is a theorem of Hilbert that a unique (up to an appropriate normalization) conformal map from any multiply 
     connected domain to a parallel slit domain exists. In the case of the annulus, from Hilbert's theorem, we can
     construct a conformal mapping $f$ of the annulus onto the plane with two horizontal slits; by an overall 
     translation, we can arrange for one of the slits to lie on the positive real axis. If $\infty$ is the image of a 
     point on the real axis, then $\overline{f(\bar{z})} = f(z)$; hence, the second slit must also lie on the real axis.
     More generally, there is a direct formula in terms of the Green's function and Schwarz 
     kernel of the domain of a mapping of any $n$-connected domain onto an annulus with $n-2$ concentric slits. Such a 
     formula would be a useful tool in considerations of cases with more connected components. See \cite{K2,K3},  for further
     information.
 \end{remark}
 
 \begin{remark}
     The function $f_{AB}(w)$ is a version of a function known as the \textit{Schottky-Klein prime function} for the torus. Such 
     functions admit explicit representations as ratios of theta functions, via the relation
    \begin{equation} \label{theta-identity}
       P(w) = -\frac{ie^{-\tau/2}}{Cr^{1/4}} \Theta_{1}(i\tau/2,r),
     \end{equation}
     where $\tau = -\log w$, and $\Theta_1$ is the first Jacobi theta function (and constant $C$, although for our purposes, this 
     constant is irrelevant). For more general properties of such functions, we refer the reader to \cite{DC2}, or the excellent
     monograph written recently by D. Crowdy \cite{DC-20}.
 \end{remark}

 Finally, we can conclude that
 \begin{prop} \label{torus-prop}
     Suppose $Q$ is a quadratic differential on the annulus $\mathbb{A}$, with first order poles at
     $B_i \in \mathbb{A}$, $i = 1,...,n$. Then, $Q$ is
     necessarily of the form 
     \begin{equation}
         Q = -C^2\prod_{i = 1}^n f_{A_i B_i} (z) \frac{dz^2}{z^2},
     \end{equation}
     For some positive constant $C^2>0$, and $A_i \in \mathbb{A}$ such that $\prod_i (A_i/B_i) > 0$.
 \end{prop}
The proof of this proposition is identical to that of the previous one. 
 In particular, it will be convenient for us to choose $\prod_i (A_i/B_i) = 1$. We now proceed to discuss the case of the doubly connected droplet. Suppose the droplet is doubly connected, and let
 $\Omega$ denote its exterior. Suppose $F$ is a function with a single simple pole at infinity, which is in class 
 $E^1$ near the boundary of $\Omega$. Since $\Omega$ is doubly connected, there exists a conformal map $\varphi : \mathbb{A} \to 
 \Omega$, for some annulus $\mathbb{A} = \{z\mid r < |z| < 1\}$, $0<r<1$; without loss of generality, $\varphi(-x) = \infty$, for some 
 $r < x < 1$. As before, we consider the quadratic differential $Q = F^2(z) dz^2$ on the complement to the droplet; 
 pulling back to the annulus, we see that $Q$ is a quadratic differential with a single simple pole in $\mathbb{A}$ at $-x$, 
 and is positive on $\partial \mathbb{A}$; thus, by Proposition \eqref{torus-prop}, we find that
     \begin{equation}\label{torus-QD}
         Q = F^2(\varphi(w)) \varphi'(w)^2 dw^2 = - f_{A_1,-x}^2(w) f_{A_2,-x}^2(w) f_{A_3,-x}^2(w) \frac{dw^2}{w^2}.
     \end{equation}
 Thus, we can conclude that
     \begin{equation}
         |w \varphi'(w)| = \frac{|P(\bar{A_1}w)P(\bar{A_2}w)P(\bar{A_3}w)|^2}{|P(-xw)|^6} \qquad \text{a.e. } \partial \mathbb{A}.
     \end{equation}
 As before, we can recover $\varphi'$ from its a.e. boundary values:
     \begin{equation} \label{2conn-phiprime}
         \varphi'(w) = \frac{1}{w}\frac{P(\bar{A_1}w)^2P(\bar{A_2}w)^2P(\bar{A_3}w)^2}{P(-xw)^6}
     \end{equation}
Since we are only looking for a particular solution, we will choose $A_1 = -x$, $A_2 = -xe^{\pi i/3}$, and 
$A_3 = -xe^{-\pi i/3}$, which clearly satisfy the constraint equation $\prod_i (A_i/B_i) = 1$. Equation \eqref{2conn-phiprime} then reads
    \begin{equation}
         \varphi'(w) = \frac{1}{w}\frac{P(-xe^{-\pi i/3}w)^2P(-xe^{\pi i/3}w)^2}{P(-xw)^4}.
    \end{equation}
Using equation \eqref{theta-identity}, we can rewrite the above
in the variable $\tau = -\log w$ as a ratio of theta functions:
    \begin{align}
        \varphi'(w) &=  e^{-\tau}\frac{\Theta_1(i\tau/2 + \frac{2\pi }{3} - i/2\log x,r)^2\Theta_1(i\tau/2 + \frac{\pi}{3} - i/2\log x,r)^2}{\Theta_1(i\tau/2 - i/2\log x +\pi/2,r)^4}\\
        & =e^{-\tau}\left[ \frac{\Theta_1\left(\frac{i}{2}(-\tau+\log{x}),r\right)}{\Theta_2\left(\frac{i}{2}(-\tau+\log{x}),r\right)}\right]^4.
    \end{align}
(The last expression is an identity satisfied by $\Theta_1$). Thus, $\varphi$ (in the variable $\tau$) is
    \begin{equation} \label{theta-integral}
        \varphi(\tau) = -\int^\tau \left[ \frac{\Theta_1\left(\frac{i}{2}(-\tau+\log{x}),r\right)}{\Theta_2\left(\frac{i}{2}(-\tau+\log{x}),r\right)}\right]^4 d\tau.
    \end{equation}
At this point, the authors' knowledge of theta function identities is exhausted, and we were unable to carry out the 
computation any further. However, this is not to say than an explicit formula does not exist; there is likely a classical 
formula for \eqref{theta-integral}; any suggestions on how to compute \eqref{theta-integral} would be greatly appreciated by the 
authors.
\begin{remark}
    It should be noted that, even after computing this integral, a solution to the droplet equation may still not exist, as it may 
    turn out that there is no choice of the remaining parameters for which a single-valued univalent map is defined. However, the 
    procedure of the previous section still applies: try to find a solution to the equation $F(z) = \dot{\bar{z}}$ with additional 
    poles inside the domain, and then `cancel' the poles using the Schwarz function of the domain. Computationally speaking, this 
    amounts to adding in additional factors of $f_{AB}(z)$ to the quadratic differential in \eqref{torus-QD}.
\end{remark}

The procedure in cases of multiple components should be very similar: begin with an ansatz for the connectivity of the complement of the droplet, say $g$, and write down the appropriate quadratic differential in the complement to the droplet. Then, map the complement back 
to a `canonical' domain (either the unit disc with a number of discs removed, or the plane with slits, or some other choice 
where computations can be performed; see Figure \ref{annulus2droplet}). Because the quadratic differential in question is positive on the boundary of the domain, 
it can be continued to its double, and thus defines a quadratic differential with prescribed poles on a Riemann surface of 
genus $g$. Quadratic differentials with this reality condition imposed on the boundary, can
be written down in terms of canonical factors on the Riemann surface. If we are using the model of disc with disc holes,
it is likely that such quadratic differentials can be expressed explicitly in terms of Schottky-Klein prime functions, as was
the case in the 1- and 2- connected cases. If instead we use the slit plane, since the Riemann surface we are considering carries an antiholomorphic involution\footnote{A Riemann surface carrying antiholomorphic involution is called a \textit{real Riemann surface}; 
such Riemann surfaces are well-studied in the theory of integrable systems, cf. \cite{BBEIM}, Chapter 3.}, and thus, it makes sense to consider (in the interest of finding exact solutions) the hyperelliptic case.
This is something we plan to consider in future work.

\begin{figure} 
    \centering
    \includegraphics{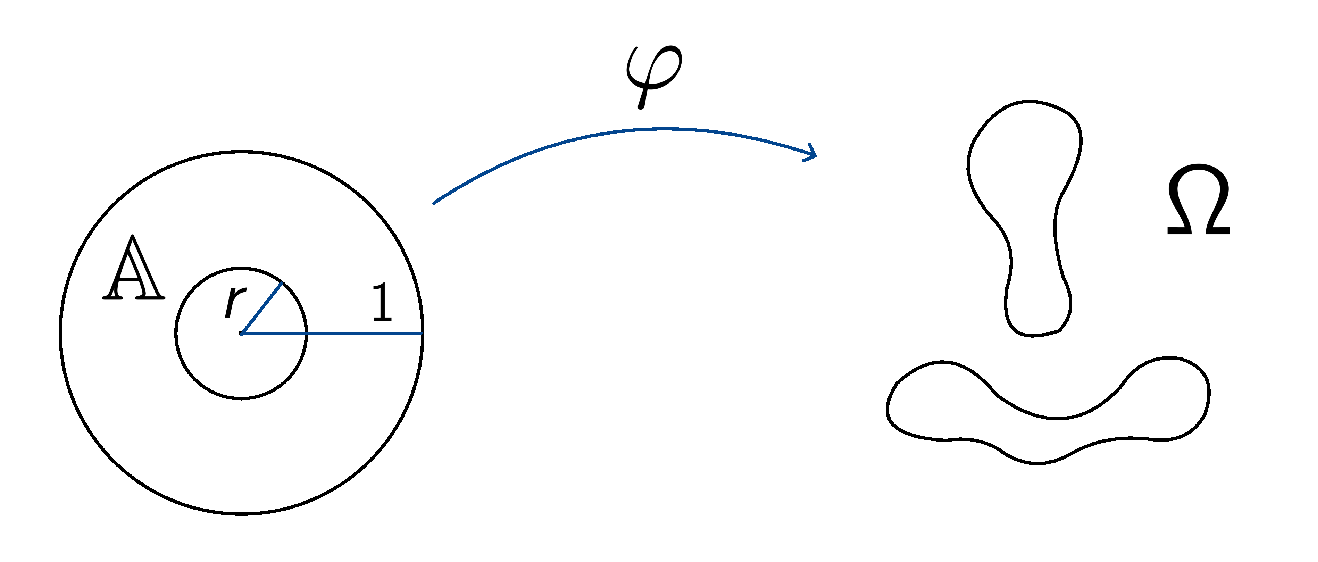}
    \caption{A map from a canonical domain, in this case the annulus, to the complement of a droplet with two connected 
    components.}
    \label{annulus2droplet}
\end{figure}

It should be noted that several solutions with multiple connected components have already been discovered by Crowdy 
\cite{DC-01,DC-99}, using the prime function mentioned above. However, these solutions are distinct, in the sense that they do 
not appear as extensions of the connected droplets derived in the previous sections of this work. It would be interesting still
to extend the connected solutions, such as the droplet in \S4, to a two-component solution, which depends continuously on the
parameter $c$. Some recent progress has been achieved \cite{DNK-20,DNK-21} in describing this transition, but there is still work
to be done; we are hopeful that our method could shed some light on this problem.

\appendix
\section{Quadratic differentials on the sphere.}\label{A}
Here, we give a brief overview of a few results on quadratic differentials that we will need in the
above text. We sacrifice maximal generality here for clearer exposition, as most of the results utilized in the text do not
require much machinery. For full details about the theory of quadratic differentials, the reader should consult
\cite{S}. 

A \textit{quadratic differential} on the Riemann sphere $\hat{\mathbb{C}}$ is an expression $Q$ of the form
	\begin{equation} \label{QD}
		Q = Q(z)dz^2,
	\end{equation}
in the local coordinate $z$ near zero, where $Q$ is meromorphic on the sphere. Under a change of coordinates
$z \to z(w)$, the quadratic differential \eqref{QD} transforms as
	\begin{equation}
		Q = Q(z)dz^2 = Q(z(w)) \left(\frac{dz}{dw}\right)^2 dw^2.
	\end{equation}
In particular, if we make the coordinate transformation $z = 1/w$,
	\begin{equation}
		Q = Q(1/w)w^{-4}dw^2,
	\end{equation}
which implies that $Q$ must be a rational function, since it must be meromorphic on the entire sphere. 
The structure of quadratic differentials is fairly restrictive; given only a little information about $Q$ on, 
say, a curve, is enough to uniquely reconstruct $Q$. This is made
evident by the following proposition, which we make reference to in \S2.
\begin{prop}
	Suppose $Q$ is a quadratic differential on $\hat{\mathbb{C}}$ with the following properties:
		\begin{enumerate}
			\item The only pole of $Q$ in $\mathbb{C}$ is at $w = 0$, where it has a pole of order $N$.
			\item $Q$ is positive definite on the circle $\{w\mid|w| = 1\}$.
		\end{enumerate}
	Then, there exist constants $c > 0$, $|A_k| < 1$, $k = 1,...,N-2$, such that
		\begin{equation}
			Q = -c \frac{dw^2}{w^2}\prod_{k=1}^{N-2}(w-A_k)(\frac{1}{w}-\bar{A}_k).
		\end{equation}
\end{prop}
	\begin{proof}
		The proof of this proposition relies only on elementary complex analysis. Let $Q = Q(w)dw^2$, for some
		rational function $Q(w)$. Assuming the hypotheses of the proposition, we note that we can write
		$Q(w) dw^2 = -P(w) w^{-N} dw^2$, for some \textit{polynomial} $P(w)$, (this follows from the fact that
		the only pole of $Q$ in $\mathbb{C}$ is at $w=0$). Now, since $-w^{-2}dw^2 = d\theta^2 >0$, we see that
		the second hypothesis of proposition implies that, for $|z| = 1$,
		$0 < -P(w) w^{-N} dw^2 = P(w) w^{-N+2} d\theta^2$. This in turn implies that $P(w)w^{-N+2} > 0$ on the
		circle; by a very classical result of complex analysis (cf. \cite{R}, exercise 4, chapter 14, for example),
		a rational function with poles only at zero which is positive on the circle necessarily takes the form
			\begin{equation*}
				P(w)w^{-N+2} = c\prod_{k=1}^{N-2}(w-A_k)(1-\bar{A}_k w)w^{-N+2},
			\end{equation*}
		for some constants $c > 0$, $|A_k| < 1$, $k = 1,...,N-2$, i.e.,  
		$P(w) = c\prod_{k=1}^{N-2}(w-A_k)(1-\bar{A}_k w)$. Simplifying this expression completes the proof of the proposition.
	\end{proof}
In short, this proposition tells us the exact form of a quadratic differential which is positive on the circle,
and has prescribed poles. By an identical method, one can prove the following proposition:
\begin{prop}
	Suppose $Q$ is a quadratic differential on $\hat{\mathbb{C}}$ with the following properties:
		\begin{enumerate}
			\item The only poles of $Q$ in the disc are poles of order $n_j$ at points 
			$\zeta_j \in \mathbb{D}$, $j = 1,...,M$, with $N = \sum_{j=1}^M n_j$, 
			\item $Q$ is positive definite on the circle $\{w\mid|w| = 1\}$.
		\end{enumerate}
	Then, there exist constants $c > 0$, $|A_k| < 1$, $k = 1,...,N-2$, such that
		\begin{equation}
			Q = -c\prod_{k=1}^{N-2}(w-A_k)(1-\bar{A}_k w) 
			\prod_{j=1}^M(w-\zeta_j)^{-n_j}(1-\bar{\zeta}_j w)^{-n_j} dw^2.
		\end{equation}
\end{prop}

\section*{Acknowledgements.} 
We would like to thank Dmitry Khavinson for introducing this problem to us, and for his helpful remarks while 
editing this work. We also offer our sincerest gratitude to the referees for their extremely valuable comments, and for 
providing a number of informative references.

\end{document}